\RequirePackage{fix-cm}
\documentclass[smallextended]{svjour3}       
\smartqed  
\usepackage{graphicx}
\usepackage{amsmath}
\usepackage{amsfonts}
\usepackage{amssymb}
\usepackage{mathtools}
\usepackage{tikz}
\usetikzlibrary{calc, math, arrows, arrows.meta}

\spnewtheorem{observation}{Observation}{\bf}{\it}
\newcommand{\ProblemA}{(MFNI)}
\newcommand{\Problem}{(BMFNI)}
\DeclareMathOperator{\val}{val}
\DeclareMathOperator{\VAL}{VAL}
\begin{document}

\title{On the bicriterion maximum flow network interdiction problem
}


\author{Luca E. Sch\"afer\(^*\)\thanks{\(^*\)Corresponding author} \and Stefan Ruzika \and Sven~O. Krumke \and Carlos M. Fonseca
}


\institute{L. E. Sch\"afer \and Stefan Ruzika \and Sven O. Krumke \at
              Department of Mathematics, Technische Universit\"at Kaiserslautern, 67663 Kaiserslautern, Germany\\
              \email{luca.schaefer@mathematik.uni-kl.de}           
           \and
           Carlos M. Fonseca \at
           CISUC, Department of Informatics Engineering, University of Coimbra, P-3030 290 Coimbra, Portugal 
}

\date{Received: date / Accepted: date}

\maketitle

\begin{abstract}
This article focuses on a biobjective extension of the maximum flow network interdiction problem, where each arc in the network is associated with two capacity values. Two maximum flows from a source to a sink are to be computed independently of each other with respect to the first and second capacity function, respectively, while an interdictor aims to minimize the value of both maximum flows by interdicting arcs. We show that this problem is intractable and that the decision problem, which asks whether or not a feasible interdiction strategy is efficient, is \(\mathcal{NP}\)-complete.
We propose a pseudopolynomial time algorithm in the case of two-terminal series-parallel graphs and positive integer-valued interdiction costs. We extend this algorithm to a fully polynomial-time approximation scheme for the case of unit interdiction costs by appropriately partitioning the objective space.
\end{abstract}

\section{Introduction}\label{sec:intro}
The maximum flow problem is one of the best studied optimization problems in operations research with several applications, cf.~\cite{Ahuja,ford1956maximal}, where one aims to send as many units of flow from a source to a sink over a network while satisfying given arc capacities. 

In the maximum flow network interdiction problem \ProblemA, an addtional opposing force, called the interdictor, is introduced, who tries to reduce the maximum flow value in a network as much as possible by interdicting arcs. Thereby, each arc is associated with an interdiction cost and the interdictor is constrained by a given interdiction budget. Although \ProblemA\ has been investigated quite early, cf. \cite{wollmer1964removing,mcmasters1970optimal,ghare1971optimal,ratliff1975finding}, the probably most prominent work considering the complexity, different integer programming formulations as well as several variants of \ProblemA\ is revealed in \cite{wood1993deterministic}. 
While some of the above mentioned articles consider special cases of \ProblemA, where e.g. each arc requires exactly one unit of the interdictor's budget, the work presented in \cite{wood1993deterministic} covers \ProblemA\ in its most general setting. In \cite{altner2010maximum}, the authors investigate the integer linear programming formulation of \ProblemA\ presented in \cite{wood1993deterministic}, provide new valid inequalities and investigate the approximability of the problem.
Further, \ProblemA\ has received considerable interest regarding the modeling of real-world applications, see e.g., drug interdiction, cf. \cite{wood1993deterministic}, hospital infection control, cf. \cite{assimakopoulos1987network}, and protection of electrical power grids, cf. \cite{salmeron2009worst}, \cite{salmeron2014value}. Furthermore, interdiction problems are often modelled as bilevel mixed integer programs, see \cite{doi:10.1002/9780470400531.eorms0932}, \cite{sinha2017review} for an overview.

It has been shown that \ProblemA\ with unit interdiction costs, also called the \(k\) most vital link problem, cf. \cite{ratliff1975finding}, is strongly \(\mathcal{NP}\)-hard, cf. \cite{phillips1993network,wood1993deterministic}. Thus, \(\mathcal{NP}\)-hardness for \ProblemA\ in its general form follows. However, several authors provide specialized algorithms for different graph classes, see e.g. \cite{phillips1993network} for a pseudopolynomial time algorithm solving \ProblemA\ on planar networks 
and \cite{wollmer1964removing} for a polynomial time algorithm for the \(k\) most vital link problem on source-sink-planar networks.

Nevertheless, literature on approximation algorithms for \ProblemA\ is rather sparse, cf. \cite{chestnut2017hardness}. In \cite{phillips1993network}, the author presents a fully polynomial time approximation scheme for planar graphs, while in \cite{chestnut2017hardness} an approximation algorithm with an approximation ratio of \(2(n-1)\) for general graphs is developed, where \(n\) denotes the number of vertices in the network.
Further, in \cite{burch2003decomposition}, an algorithm is proposed that either returns a \((1,1+\frac{1}{\varepsilon})\)-approximation, i.e., a feasible solution that does not deviate from the optimal solution by more than a factor of \(1+\frac{1}{\varepsilon}\), or a \((1+\varepsilon,1)\)-pseudoapproximation, i.e., a solution that violates the interdiction budget by at most a factor of \(1+\varepsilon\), but with the same objective function value as the optimal solution.

In contrast, there are only a few articles dealing with multiobjective interdiction problems. In \cite{royset2007solving}, the authors consider a biobjective extension of \ProblemA, where an attacker aims at minimizing the maximum flow and the total interdiction costs. An evolutionary algorithm for a multiobjective variant of \ProblemA\ is developed in \cite{ramirez2010bi}, see also \cite{rocco2011assessing}.

\paragraph{Our contribution}
In this article, we consider a new biobjective extension of \ProblemA, called the biobjective maximum flow network interdiction problem \Problem, which to the best of our knowledge has not been considered in the literature so far. We associate each arc with two integer capacity values such that two maximum flows can independently be computed from a source to a sink. An opposing force, called the interdictor, aims at reducing those maximum flows simultaneously by interdicting arcs. Consequently, this leads to several incommensurable interdiction strategies.

The remainder of this article is outlined as follows. In Section~\ref{sec:preliminaries}, we formally introduce the problem setting. We address the complexity of \Problem\ in Section \ref{sec:complexity} by providing an instance with an exponential number of non-dominated points and showing that its hard to decide whether or not a given interdiction strategy is efficient. In Section~\ref{sec:solution}, we propose a pseudopolynomial time algorithm for \Problem\ on two-terminal series-parallel graphs. We extend this algorithm to a fully polynomial time approximation scheme for the same problem with unit interdiction costs and, thus, reduce the gap on approximation algorithms for interdiction problems. Section~\ref{sec:conclusion} summarizes the paper and proposes further directions of research.

Note that a preliminary version of this article has been published in \cite{preliminaryversion}.
In comparison to the preliminary version, all proofs to the respective theorems have been added. Further, the connection to the biobjective knapsack problem for a special version of \Problem\ has been elaborated. Additionally, in case of two-terminal series-parallel graphs and unit interdiction costs, we extended the dynamic programming algorithm to a fully polynomial time approximation scheme.

\section{Preliminaries and problem formulation}\label{sec:preliminaries}
Let \(G=(V,A)\) be a directed graph with vertex set \(V\) and arc set \(A\), where \(s, t\in V\)  with \(s\neq t\) denote the source and sink vertex in \(G\), respectively. Further, we set \(n\coloneqq|V|\) and \(m\coloneqq|A|\).
An \(s\)-\(t\)-flow is a function \(f\colon A \rightarrow \mathbb{R}_+\) assigning a flow value to each arc while satisfying flow conservation constraints \(\sum_{a \in A: a \in \delta^-(v)} f(a)-\sum_{a \in A: a \in \delta^+(v)} f(a)=0\) for all \(v \in V\backslash\{s,t\}\) with \(\delta^-(v)\) and \(\delta^+(v)\) denoting the set of incoming and outgoing arcs of \(v \in V\), respectively. We call \(f\) feasible, if \(f(a)\leq u(a)\) for all \(a \in A\) for some capacity function \(u\colon~A~\rightarrow~\mathbb{N}\). The value of \(f\) is equal to the excess at the sink vertex \(t\), i.e., \(\val(f) = \sum_{a \in A: a \in \delta^-(t)} f(a)-\sum_{a \in A: a \in \delta^+(t)} f(a)\). The maximum flow problem asks  for the maximum flow value over all feasible \(s\)-\(t\)-flows, denoted by \(\VAL(G,u)\), cf.~\cite{Ahuja}. 

As in \ProblemA, we introduce an interdictor who is constrained by an interdiction budget \(B\in\mathbb{N}\), while each arc \(a\in A\) is associated with an interdiction cost \(c(a)\in\mathbb{N}\). Consequently, the set of all feasible interdiction strategies, denoted by \(\Gamma\), can be expressed as follows:
\begin{equation*}
\Gamma \coloneqq \left\{\gamma = (\gamma_a)_{a \in A} \in \{0,1\}^m \mid \sum\limits_{a \in A} c(a)\cdot\gamma_a \leq B\right\},	
\end{equation*}
where \(\gamma_a\) equals zero or one, if arc \(a\) is interdicted or not, respectively.
Further, we assign two capacity values to each arc, i.e., \(u\coloneqq(u^1,u^2)\colon A \rightarrow \mathbb{N}^2\) with \(u^i\colon A\rightarrow \mathbb{N}, i=1,2\). By \(U^1\) and \(U^2\), we denote the maximum arc capacity with respect to \(u^1\) and \(u^2\), respectively, i.e., \(U^i\coloneqq\max\{u^i(a)\mid a\in A\}\) for \(i=1,2\). Further, \(U\) denotes the maximum of \(U^1\) and \(U^2\).  Thus, two feasible \(s\)-\(t\)-flows \(f^1\) and \(f^2\), denoted by \(f\coloneqq(f^1,f^2)\), can be computed in \(G\) with respect to \(u^1\) and \(u^2\), respectively. In what follows, each interdiction strategy \(\gamma \in \Gamma\) induces an interdicted graph \(G(\gamma)\coloneqq(V', A')\) with \(V'=V\) and \(A'=A\setminus A(\gamma)\), where \(A(\gamma)\coloneqq\{a\in A\mid \gamma_a=1\}\). Thus, for a given interdiction strategy \(\gamma\in\Gamma\), we denote by \(\VAL(G(\gamma),u)\coloneqq \left(\VAL(G(\gamma),u^1),\VAL(G(\gamma),u^2)\right)\) the vector of maximum flow values in the interdicted graph \(G(\gamma)\) with respect to \(u^1\) and \(u^2\), respectively. Note that those two maximum flows can be computed in polynomial time, cf. \cite{Ahuja}. Since every interdiction strategy leads to a vector of maximum flows, we use the following orders on \(\mathbb{N}^2\), which are commonly used in the field of multi-, or more specifically, in biobjective optimization, cf. \cite{ehrgott2005multicriteria}:
\begin{align*}
y^1 \leq y^2 & \Leftrightarrow y^1_k \leq y^2_k \text{ for } k = 1,2 \text{ and } y^1 \neq y^2,\\
y^1 \leqq y^2 & \Leftrightarrow y^1_k \leq y^2_k \text{ for } k = 1,2.
\end{align*}
Generally speaking, in biobjective optimization problems, one aims to find those feasible solutions that do not allow to improve the one objective function without deteriorating the other, which leads to the following definition.
\begin{definition}\label{def:dominance}
	A feasible interdiction strategy \(\gamma \in \Gamma\) is called \emph{efficient}, if there does not exist \(\gamma' \in \Gamma\) such that
	\begin{equation*}
	\VAL(G(\gamma'),u) \leq \VAL(G(\gamma),u).
	\end{equation*}
	In this case, we call \(\VAL(G(\gamma),u)\) a \emph{non-dominated} point. By \(\Gamma_E\) and \(Z_N\), we denote the set of efficient interdiction strategies and non-dominated points, respectively.
\end{definition}
Using Definition \ref{def:dominance}, we state \Problem\ as \(\min\limits_{\gamma\in\Gamma}\VAL(G(\gamma),u)\).

Further, we say a feasible interdiction strategy \(\gamma\in\Gamma\) \(\varepsilon\)-approximates another interdiction strategy \(\gamma'\in\Gamma\) for some \(\varepsilon>0\), if \[\VAL(G(\gamma),u)~\leqq~(1+\varepsilon)~\VAL(G(\gamma'),u).\] Additionally, we call \(\mathcal{A}\subseteq \Gamma\) an \(\varepsilon\)-approximation of the set of non-dominated points, if for every \(\gamma\in\Gamma\) there exists an    \(a\in\mathcal{A}\) that \(\varepsilon\)-approximates \(\gamma\). We call an algorithm a fully polynomial-time approximation scheme (FPTAS) for \Problem, if for any instance of \Problem\ and for any value \(\varepsilon\in\mathbb{Q}_+\), the algorithm returns an \(\varepsilon\)-approximation in time polynomial both in the size of the instance and in \(\frac{1}{\varepsilon}\).

In what follows, we focus on two-terminal series-parallel graphs, which are, due to \cite{eppstein1992parallel}, defined as follows.
\begin{definition}
	A directed graph \(G=(V,A)\) is called two-terminal series-parallel with source \(s\) and sink \(t\), if \(G\) can be constructed by a sequence of the following operations.
	\begin{itemize}
		\item[1)] Construct a primitive graph \(G'=(V',A')\) with \(V'=\{s,t\}\) and \(A'=\{(s,t)\}\).
		\item[2)] (Parallel Composition) Given two directed, series-parallel graphs \(G_1\) with source \(s_1\) and sink \(t_1\) and \(G_2\) with source \(s_2\) and sink \(t_2\), form a new graph \(G\) by identifying \(s=s_1=s_2\) and \(t=t_1=t_2\).
		\item[3)] (Series Composition) Given two directed, series-parallel graphs \(G_1\) with source \(s_1\) and sink \(t_1\) and \(G_2\) with source \(s_2\) and sink \(t_2\), form a new graph \(G\) by identifying \(s=s_1\), \(t_1=s_2\) and \(t_2=t\).
	\end{itemize}
\end{definition}
Two-terminal series-parallel graphs can be recognized in polynomial time along with the corresponding decomposition tree. We denote the decomposition tree of a two-terminal series-parallel graph \(G\) by \(T_G\). The size of \(T_G\) is linear in the size of \(G\), cf. \cite{valdes1982recognition}. Note that \(T_G\) specifies how \(G\) has been constructed by using the above mentioned rules. Thus, each vertex in \(T_G\) can be associated with a two-terminal series-parallel graph itself, see Figure \ref{fig:decomposition_tree}. Consequently, if we refer to a graph \(H\) in \(T_G\), we actually refer to the graph \(H\), which actually denotes a subgraph of \(G\), corresponding to a vertex in \(T_G\).
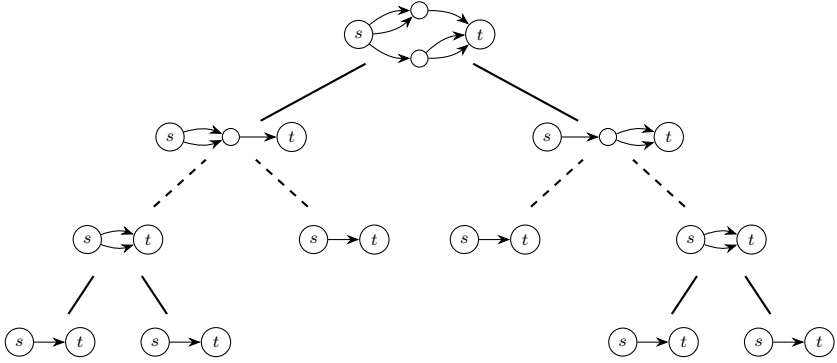
\begin{figure}[htbp]
	\begin{center}
		\begin{tikzpicture}[scale=0.8, transform shape]
		\tikzmath{
			\smallx = 1;
			\smally = 0.4;
			\basex = 3.1;
			\basey = 1.7;
			\factor = 0.6;
			\gap = 0.0;
			\bigBend = 40;
			\smallBend = 20;
			\rootx = 0;
			\rooty = 0;
			\rootlx = \rootx - \basex;
			\rootrx = \rootx + \basex;
			\leveloney = \rooty - \basey;
			\rootllx = \rootlx - \factor*\basex;
			\rootlrx = \rootlx + \factor*\basex;
			\rootrlx = \rootrx - \factor*\basex;
			\rootrrx = \rootrx + \factor*\basex;
			\leveltwoy = \rooty - 2*\basey;
			\rootlllx = \rootllx - \factor*\factor*\basex;
			\rootllrx = \rootllx + \factor*\factor*\basex;
			\rootrllx = \rootrlx - \factor*\factor*\basex;
			\rootrlrx = \rootrlx + \factor*\factor*\basex;
			\rootrrrx = \rootrrx + \factor*\factor*\basex;
			\rootrrlx = \rootrrx - \factor*\factor*\basex;
			\levelthreey = \rooty - 3*\basey;
			\rightshift = 0.0*\smallx;
			\radius = 0.55;
		}
		
		\draw[thick] (\rootx, \rooty) -- (\rootlx, \leveloney);
		\draw[thick] (\rootx, \rooty) -- (\rootrx, \leveloney);
		
		\draw[thick, dashed] (\rootlx, \leveloney) -- (\rootllx, \leveltwoy);
		\draw[thick, dashed] (\rootlx, \leveloney) -- (\rootlrx, \leveltwoy);
		
		\draw[thick, dashed] (\rootrx, \leveloney) -- (\rootrlx, \leveltwoy);
		\draw[thick, dashed] (\rootrx, \leveloney) -- (\rootrrx, \leveltwoy);
		
		\draw[thick] (\rootllx, \leveltwoy) -- (\rootlllx, \levelthreey);
		\draw[thick] (\rootllx, \leveltwoy) -- (\rootllrx, \levelthreey);
		
		\draw[thick] (\rootrrx, \leveltwoy) -- (\rootrrrx, \levelthreey);
		\draw[thick] (\rootrrx, \leveltwoy) -- (\rootrrlx, \levelthreey);

		\filldraw [white] (\rootx,\rooty) circle (1);
		\filldraw [white] (\rootlx,\leveloney) circle (\radius);
		\filldraw [white] (\rootrx,\leveloney) circle (\radius);
		\filldraw [white] (\rootllx,\leveltwoy) circle (1.3*\radius);
		\filldraw [white] (\rootlrx,\leveltwoy) circle (1.3*\radius);
		\filldraw [white] (\rootrlx,\leveltwoy) circle (1.3*\radius);
		\filldraw [white] (\rootrrx,\leveltwoy) circle (1.3*\radius);
		\filldraw [white] (\rootlllx,\levelthreey) circle (\radius);
		\filldraw [white] (\rootllrx,\levelthreey) circle (\radius);
		\filldraw [white] (\rootrrrx,\levelthreey) circle (\radius);
		\filldraw [white] (\rootrrlx,\levelthreey) circle (\radius);

		\begin{scope}[every node/.style={circle,draw, fill=white}]
		\node (S) at  ( \rootx-1*\smallx, 0) {$s$};
		\node (v1) at  ( \rootx, \smally) {};
		\node (v2) at  ( \rootx, -\smally) {};
		\node (T) at  ( \rootx+1*\smallx, 0) {$t$};
		\end{scope}
		\begin{scope}[>={Stealth[black]}, sloped]
		\path [->] (S) [bend left=\smallBend]edge node [below] {} (v1);
		\path [->] (S) [bend right=\smallBend]edge node [below] {} (v1);
		\path [->] (S) [bend right=\smallBend]edge node [below] {} (v2);
		\path [->] (v2) [bend right=\smallBend]edge node [below] {} (T);
		\path [->] (v2) [bend left=\smallBend]edge node [below] {} (T);
		\path [->] (v1) [bend left=\smallBend] edge node [above] {} (T);
		\end{scope}
		
		\begin{scope}[every node/.style={circle,draw, fill=white}]
		\node (S) at  (-1*\smallx+\rootlx, 0+\leveloney) {$s$};
		\node (v) at  ( 0*\smallx+\rootlx, 0+\leveloney) {};
		\node (T) at  ( 1*\smallx+\rootlx, 0+\leveloney) {$t$};
		\end{scope}
		\begin{scope}[>={Stealth[black]}, sloped]
		\path [->] (v) [bend left=0]edge node [below] {} (T);
		\path [->] (S) [bend left=\smallBend] edge node [above] {} (v);
		\path [->] (S) [bend right=\smallBend] edge node [above] {} (v);
		\end{scope}
		
		\begin{scope}[every node/.style={circle,draw, fill=white}]
		\node (S) at  (-1*\smallx+\rootrx, 0+\leveloney) {$s$};
		\node (v) at  ( 0*\smallx+\rootrx, 0+\leveloney) {};
		\node (T) at  ( 1*\smallx+\rootrx, 0+\leveloney) {$t$};
		\end{scope}
		\begin{scope}[>={Stealth[black]}, sloped]
		\path [->] (S) [bend left=0]edge node [below] {} (v);
		\path [->] (v) [bend left=\smallBend] edge node [above] {} (T);
		\path [->] (v) [bend right=\smallBend] edge node [above] {} (T);
		\end{scope}
		
		\begin{scope}[every node/.style={circle,draw, fill=white}]
		\node (S) at  (-0.5*\smallx+\rootllx, \leveltwoy) {$s$};
		\node (T) at  ( 0.5*\smallx+\rootllx, \leveltwoy) {$t$};
		\end{scope}
		\begin{scope}[>={Stealth[black]}, sloped]
		\path [->] (S) [bend right=\smallBend]edge node [below] {} (T);
		\path [->] (S) [bend left=\smallBend]edge node [below] {} (T);
		\end{scope}
		
		\begin{scope}[every node/.style={circle,draw, fill=white}]
		\node (S) at  (-0.5*\smallx+\rootlrx, \leveltwoy) {$s$};
		\node (T) at  ( 0.5*\smallx+\rootlrx, \leveltwoy) {$t$};
		\end{scope}
		\begin{scope}[>={Stealth[black]}, sloped]
		\path [->] (S) [bend right=0]edge node [below] {} (T);
		\end{scope}
		
		\begin{scope}[every node/.style={circle,draw, fill=white}]
		\node (S) at  (-0.5*\smallx+\rootrlx, 0+\leveltwoy) {$s$};
		\node (T) at  ( 0.5*\smallx+\rootrlx, 0+\leveltwoy) {$t$};
		\end{scope}
		\begin{scope}[>={Stealth[black]}, sloped]
		\path [->] (S) [bend right=0]edge node [below] {} (T);
		\end{scope}
		
		\begin{scope}[every node/.style={circle,draw, fill=white}]
		\node (S) at  (-0.5*\smallx+\rootrrx, 0+\leveltwoy) {$s$};
		\node (T) at  ( 0.5*\smallx+\rootrrx, 0+\leveltwoy) {$t$};
		\end{scope}
		\begin{scope}[>={Stealth[black]}, sloped]
		\path [->] (S) [bend right=\smallBend]edge node [below] {} (T);
		\path [->] (S) [bend left=\smallBend]edge node [below] {} (T);
		\end{scope}
		
		\begin{scope}[every node/.style={circle,draw, fill=white}]
		\node (S) at  (-0.5*\smallx+\rootlllx, \levelthreey) {$s$};
		\node (T) at  ( 0.5*\smallx+\rootlllx, \levelthreey) {$t$};
		\end{scope}
		\begin{scope}[>={Stealth[black]}, sloped]
		\path [->] (S) [bend left=0]edge node [below] {} (T);
		\end{scope}
		
		\begin{scope}[every node/.style={circle,draw, fill=white}]
		\node (S) at  (-0.5*\smallx+\rootllrx, \levelthreey) {$s$};
		\node (T) at  ( 0.5*\smallx+\rootllrx, \levelthreey) {$t$};
		\end{scope}
		\begin{scope}[>={Stealth[black]}, sloped]
		\path [->] (S) [bend right=0]edge node [below] {} (T);
		\end{scope}

		\begin{scope}[every node/.style={circle,draw, fill=white}]
		\node (S) at  (-0.5*\smallx+\rootrrlx, \levelthreey) {$s$};
		\node (T) at  ( 0.5*\smallx+\rootrrlx, \levelthreey) {$t$};
		\end{scope}
		\begin{scope}[>={Stealth[black]}, sloped]
		\path [->] (S) [bend right=0]edge node [below] {} (T);
		\end{scope}
		
		\begin{scope}[every node/.style={circle,draw, fill=white}]
		\node (S) at  (-0.5*\smallx+\rootrrrx, \levelthreey) {$s$};
		\node (T) at  ( 0.5*\smallx+\rootrrrx, \levelthreey) {$t$};
		\end{scope}
		\begin{scope}[>={Stealth[black]}, sloped]
		\path [->] (S) [bend right=0]edge node [below] {} (T);
		\end{scope}
		\end{tikzpicture}
	\end{center}
	\caption{Decomposition tree \(T_G\) of a two-terminal series-parallel graph \(G\). The root vertex corresponds to \(G\) itself. Every leaf corresponds to a primitive graph. Dashed lines correspond to series compositions, whereas straight lines correspond to parallel compositions.}
	\label{fig:decomposition_tree}
\end{figure}

If a graph \(G=(V,A)\) consists of only two vertices, i.e., \(V=\{s,t\}\), connected by \(m\) parallel arcs from \(s\) to \(t\), we call \(G\) a two-terminal parallel graph. Note that a two-terminal parallel graph is a special case of a two-terminal series-parallel graph.

In the following, an instance of \Problem\ is denoted by \((G,u,c,B)\), where \(G=(V,A)\) denotes a directed graph, \(u=(u^1,u^2)\) assigns two capacities to each arc, \(c\) associates every arc with an interdiction cost and \(B\) refers to the interdiction budget.

\section{Complexity Results}\label{sec:complexity}
In this section, we prove the following decision version of \Problem\ to be \(\mathcal{NP}\)-complete: Given an instance \((G,u,c,B)\) of (BMFNI) and a value \(K~= (K^1,K^2)^\top \in \mathbb{N}^2\), decide whether there exists an interdiction strategy \(\gamma \in \Gamma\) with \(\VAL(G(\gamma),u)\leq K\). Thus, we basically show that its hard to decide whether or not an interdiction strategy is efficient.
\begin{theorem}\label{thm:npc}
	The decision version of \Problem\ is \(\mathcal{NP}\)-complete, even for unit interdiction costs, i.e., \(c(a)=1\) for all \(a \in A\), and even on two-terminal parallel graphs.
\end{theorem}
\begin{proof}
	The problem is clearly contained in \(\mathcal{NP}\).
	To show \(\mathcal{NP}\)-completeness, we conduct a polynomial time reduction from the binary knapsack problem, which is known to be \(\mathcal{NP}\)-complete, cf. \cite{garey2002computers}.
	The decision version of the binary knapsack problem is as follows: Given a finite set \(I =\{1,\ldots,n\}\) of elements with profits \(p_i \in \mathbb{N}\) and weights \(w_i \in \mathbb{N}\) for all elements \(i \in I\) and positive integers \(P\) and \(W\), does there exist a solution \(I^*\subseteq I\) such that \(\sum_{i \in I^*} w_i \leq W\) and \(\sum_{i \in I^*} p_i \geq P\)?
	
	Given an instance of the binary knapsack problem, we construct an instance of \Problem. Let \(G=(V,A)\) with \(V = \{s,t\}\) and for each element \(i\in I\) introduce an arc \(a_i\) and a dummy arc \(b_i\) going from \(s\) to \(t\). Thus, \(A\coloneqq A_1\cup A_2\) with \(A_1=\{a_1,\ldots,a_n\}\) and \(A_2=\{b_1,\ldots,b_n\}\). Further, we set \(u(a_i)=(p_{\max}-p_i, w_i)\) and \(u(b_i) = (p_{\max}, 0)\) with \(p_{\max}\coloneqq \max_{i \in I}\{p_i\}\). We define \(B\coloneqq n\) and \(K=(K^1,K^2)\coloneqq (np_{\max}-P,W)\).
	
	Given a solution \(I^*\subseteq I\) of the binary knapsack instance, i.e., \(\sum_{i\in I^*}p_i \geq P\) and \(\sum_{i\in I^*}w_i\leq W\), we construct an interdiction strategy \(\gamma\coloneqq (\gamma_a,\gamma_b) \coloneqq (\gamma_{a_1},\ldots,\gamma_{a_n},\gamma_{b_1},\ldots,\gamma_{b_n})\) as follows:
	\begin{equation*}
	\gamma_{a_i} = \begin{cases}
	1, &\text{ if } i\notin I^*\\
	0, & \text{ else}
	\end{cases}\hspace{2cm}
	\gamma_{b_i} = \begin{cases}
	1, &\text{ if } i \in I^*\\
	0, & \text{ else.}
	\end{cases}
	\end{equation*}
	It follows:
	\begin{align*}
	\VAL(G(\gamma),u^1) &= \sum_{i\mid a_i\in A\setminus A(\gamma_a)} u^1(a_i) + \sum_{i\mid b_i\in A\setminus A(\gamma_b)} u^1(b_i)\\
	&= \sum_{i\mid a_i\in A\setminus A(\gamma_a)} p_{\max}-p_i + \sum_{i\mid b_i\in A\setminus A(\gamma_b)} p_{\max}\\
	&= np_{\max} - \sum_{i\mid a_i\in A\setminus A(\gamma_a)} p_i \\
	&\leq np_{\max} - P = K^1\\
	\VAL(G(\gamma),u^2) &= \sum_{i\mid a_i\in A\setminus A(\gamma_a)} u^2(a_i) + \sum_{i\mid b_i\in A\setminus A(\gamma_b)} u^2(b_i)\\ 
	&= \sum_{i\mid a_i\in A\setminus A(\gamma_a)} w_i + \sum_{i\mid b_i\in A\setminus A(\gamma_b)} 0
	\leq W = K^2.
	\end{align*}
	For the other direction let \(\gamma = (\gamma_a,\gamma_b)\in\Gamma\) be a feasible interdiction strategy. Without loss of generality, we can assume that \(\gamma\) either interdicts \(a_i\) or \(b_i\) for all \(i=1,\ldots,n\), since all \(b_i\) have the same capacity values.
	Thus, we set \(I^*\coloneqq\{i\mid \gamma_{a_i}\neq 1, i=1,\ldots,n\}\) and show that \(I^*\) is a solution to the knapsack problem:
	\begin{align*}
	&\sum_{i\in I^*} w_i = \sum_{i\in I^*} u^2(a_i) = \sum_{i\in I^*} u^2(a_i) + \sum_{i\notin I^*} u^2(b_i) = \VAL(G(\gamma),u^2) \leq K^2 = W\\
	&np_{\max}-\sum_{i\in I^*} p_i = \sum_{i\in I^*} p_{\max} - p_i + \sum_{i\notin I^*} p_{\max} = \sum_{i\in I^*} u^1(a_i) + \sum_{i\notin I^*} u^1(b_i)\\
	& = \VAL(G(\gamma),u^1) \leq K^1 = np_{\max} - P.
	\end{align*}
	Thus, it holds that \(\sum_{i\in I^*} w_i\leq W\) and \(\sum_{i\in I^*} p_i \geq P\), which concludes the proof.\qed
\end{proof}

Further, in multiobjective combinatorial optimization, one is usually interested in the worst-case size of the non-dominated set. To account for this question, we show that \Problem\ is intractable, i.e., there might be exponentially many non-dominated points with respect to the size of the problem instance. To prove intractability, consider the parametric cost knapsack problem, cf. \cite{burkard1995inverse}:
\begin{equation*}
f(q) = \max\left\{\sum_{i=1}^n \left(p^1_i q + p^2_i \right) x_i \mid \sum_{i=1}^n w_ix_i\leq b, x_i \in \{0,1\}, i = 1,\ldots,n\right\},
\end{equation*}
where \(p^1_i, p^2_i \in \mathbb{R}_+\) refer to the coefficients of the linear cost function \(f\) for item \(i\) and \(q\geq 0\). The weight of item \(i\) is denoted by \(w_i\) and \(b\) refers to the knapsack capacity.
It is known that the number of breakpoints, i.e., values of \(q\) where the slope of \(f(q)\) changes, is in general exponential in the number of variables.
\begin{theorem}[\cite{carstensen1983complexity}]\label{thm:knapintractable}
	For every \(n\in \mathbb{N}\), there exists a parametric cost knapsack problem with \(\frac{1}{2}\left(9n^2 - 7n\right)\) variables, such that \(f(q)\) has \(2^n-1\) breakpoints in the interval \((-2^n,2^n)\).\qed
\end{theorem}
This holds even true, when restricting \(q\) to be on a compact, positive real-valued interval, cf. \cite{giudici2017approximation}, and even for integral input data, cf. \cite{holzhauser2017fptas}.

Further, the parametric cost knapsack problem can be interpreted as a weighted sum scalarization of the following biobjective knapsack problem:
\begin{subequations}
	\begin{align*}
	\label{eq:lp0}
	\max \quad & \left(\sum_{i=1}^n p^1_ix_i, \sum_{i=1}^{n} p^2_ix_i\right) \\
	\text{s.t.} \quad &  \sum_{i=1}^n w_ix_i\leq b\\
	& x_i \in \{0,1\}, & i = 1,\ldots,n
	\end{align*}
\end{subequations}
Thus, using Theorem \ref{thm:knapintractable}, we can conclude:
\begin{corollary}\label{cor:kpintrac}
	The biobjective knapsack problem is intractable, i.e., there might be exponentially many non-dominated points with respect to the size of the problem instance.\qed
\end{corollary}

\begin{theorem}\label{thm:kpequiv}
	Let \((G,u,c,B)\) be an instance of \Problem\ with \(G\) being a two-terminal parallel graph with \(m\) arcs and capacities \((u^1(a),u^2(a))\) for all \(a~\in A(G)\). Further, let \(X_E\) denote the set of efficient solutions of 
	\begin{equation*}
	\max\left\{v(x)\coloneqq\left(\sum_{i=1}^{m}u^1(a_i)x_i, \sum_{i=1}^{m}u^2(a_i)x_i\right)\mid \sum_{i=1}^{m}c(a_i)x_i \leq B, x\in \{0,1\}^m\right\}.
	\end{equation*}
	Then, it holds that \(X_E = \Gamma_E\), where \(\Gamma_E\) denotes the set of efficient interdiction strategies of the \Problem\ instance.
\end{theorem}
\begin{proof}
	\underline{\(\Gamma_E \subseteq X_E\):} Let \(\gamma\in \Gamma_E\) be an efficient interdiction strategy and assume \(\gamma\notin X_E\). Then, there exists \(x\in X_E\) such that \(v(x)\geq v(\gamma)\). Consequently, it follows that \(\VAL(G(x),u)=v(e)-v(x)\leq v(e) - v(\gamma) = \VAL(G(\gamma),u)\) with \(e=(1,\ldots,1)\), contradicting that \(\gamma\) is an effcicient interdiction strategy. Thus, \(\gamma\in X_E\).
	
	\underline{\(X_E \subseteq \Gamma_E\):} Let \(x \in X_E\) and assume \(x \notin \Gamma_E\). Then, there exists \(\gamma\in\Gamma\) such that \(\VAL(G(\gamma),u) = v(e)-v(\gamma)\leq v(e)-v(x) = \VAL(G(x),u)\), It follows that \(v(x)\leq v(\gamma)\), contradicting that \(x\in X_E\). Thus, \(x\in \Gamma_E\).\qed
\end{proof}

Note that the optimization problem in Theorem \ref{thm:kpequiv} is nothing but a biobjective knapsack problem, where the capacities and costs of \Problem\ denote the profits and weights of the biobjective knapsack problem, respectively. Thus, using Corollary \ref{cor:kpintrac} and Theorem \ref{thm:kpequiv}, we obtain the following result.
\begin{theorem}
	The problem \Problem\ is intractable even on two-terminal parallel graphs, i.e., the number of non-dominated points might be exponential in the size of the problem instance. In fact, even the set of supported non-dominated points might be exponential in the size of the problem instance.\qed
\end{theorem}

\section{Solution Procedures}\label{sec:solution}
In this section, we briefly state how to tackle \Problem\ on two-terminal parallel graphs and propose a solution procedure to solve the problem on two-terminal series-parallel graphs.

By Theorem \ref{thm:kpequiv}, we know that \Problem\ on two-terminal parallel graphs can be formulated as a biobjective knapsack problem.

Since there is an FPTAS for the biobjective knapsack problem, cf. \cite{papadimitriou2000approximability,safer1995fast}, \Problem\ can be solved by using the same approximation scheme.
\begin{corollary}
	There is an FPTAS for \Problem\ on two-terminal parallel graphs constructing an \(\varepsilon\)-approximation of the set of non-dominated points.\qed
\end{corollary}

Next, we propose a dynamic programming algorithm for \Problem\ for the case of two-terminal series-parallel graphs. We assume a decomposition tree~\(T_G\) for a given two-terminal series-parallel graph to be given.
For a graph \(H\) in \(T_G\), we denote by \(L(H,x)\) the set of non-dominated points of \Problem\ in \(H\) using an interdiction budget of \(x \in \mathbb{N}\), i.e., the interdictor's budget is \(x\). By \(s_H\) and \(t_H\), we refer to the source and sink of \(H\), respectively. 

For the case of \(H=(V_H,A_H)\) being a primitive graph, i.e., a leaf of \(T_G\), with \(V_H=\{s_H,t_H\}\) and \(A_H=\{a^*\}\), where \(a^*=(s_H,t_H)\), we can clearly compute \(L(H,x)\) for all \(x \in \{0,1,\ldots,B\}\) in the following way:
\begin{equation}\label{eq:1}
L(H,x) =\begin{cases}
\{(u^1(a^*),u^2(a^*))\}, & \text{ if } x = 0,1,\ldots,c(a^*)-1\\
\{(0,0)\}, & \text{ if } x = c(a^*),\ldots,B
\end{cases}
\end{equation}
Note that if \(c(a^*)>B\), then \(L(H,x)\) is equal to \(\{(u^1(a^*),u^2(a^*))\}\) for all \(x\in\{0,1,\ldots,B\}\).

Now, let \(H\) be the parallel composition of \(H_1\) and \(H_2\). Then, \(L(H,x)\) can be computed by adding each non-dominated point in \(L(H_1,k)\) to every non-dominated point in \(L(H_2,x-k)\) for all \(k\in\{0,\ldots,x\}\). Afterwards, dominated points are discarded with respect to the Pareto-order, which yields
\begin{equation}\label{eq:2}
L(H,x) = \min \left\{\bigcup_{k=0}^x L(H_1,k) \oplus L(H_2,x-k)\right\} \text{ for } x = 0,1,\ldots,B,
\end{equation}
where \(\oplus\) denotes the Minkowski sum.

Let \(H\) be the series composition of \(H_1\) and \(H_2\). Analogously to above, \(L(H,x)\) can be computed by combining each non-dominated point in \(L(H_1,k)\) with every non-dominated point in \(L(H_2,x-k)\) for all \(k\in\{0,\ldots,x\}\). We combine these non-dominated points by taking the respective minima of the maximum flows in each component. Again, dominated points are discarded afterwards:
\begin{equation}\label{eq:3}
L(H,x) = \min \left\{\bigcup_{k=0}^x L(H_1,k) \odot L(H_2,x-k)\right\} \text{ for } x = 0,1,\ldots,B,
\end{equation}
where \(r \odot s \coloneqq \left(\min\{r^1,s^1\},\min\{r^2,s^2\}\right)\) for \(r,s \in \mathbb{R}^2\) with \(r=(r^1, r^2)\) and \(s=(s^1,s^2)\) and \(R\odot S\coloneqq \{r\odot s\mid r \in R, s \in S\}\) for \(R,S \subset \mathbb{R}^2\).

\begin{theorem}\label{thm:correctness}
	After termination of the dynamic programming algorithm defined by formulas \eqref{eq:1}, \eqref{eq:2} and \eqref{eq:3}, the label set \(L(G,B)\) defines the set of non-dominated points for the \Problem\ instance.
\end{theorem}
\begin{proof}
	We use induction on the size of the decomposition tree \(T_G\) of \(G\).
	Using \eqref{eq:1}, the set of non-dominated points for a primitive graph, i.e., a leaf of \(T_G\), can easily be computed.
	Now, let \(H\) be a graph in \(T_G\) and a parallel composition of \(H_1\) and \(H_2\) and let \(y = \VAL(H(\gamma),u)\) be a non-dominated point for \(H\) with a total interdiction cost of \(x^*\) for some \(x^* \in \{0,\ldots,B\}\) and \(\gamma\in\Gamma\) that has not been found. Let \(y=p+q\) with \(p=\VAL(H_1(\gamma^1),u)\) and \(q~=~\VAL(H_2(\gamma^2),u)\), where \(\gamma^1+\gamma^2=\gamma\) and let \(c=\sum_{a\in A}c(a)\cdot \gamma^1_a\). If \(p\in L(H_1,c)\) and \(q~\in~L(H_2,x^*-c)\), then \(y\) would have been added to \(L(H,x^*)\) due to construction of the algorithm.
	Thus, we may assume that \(p\notin L(H_1,c)\) or \(q\notin L(H_2,x^*-c)\). Without loss of generality, we assume that \(p\notin L(H_1,c)\). It follows that there exists a non-dominated point \(r \in L(H_1,c)\) with \(r\leq p\). Consequently, it holds that \(r+q\leq y\), which contradicts the assumption that \(y\) is non-dominated. 
	Now, let \(H\) be a series composition of \(H_1\) and \(H_2\) and let \(y=\VAL(H(\gamma),u)\) again be a non-dominated point for \(H\) with a total interdiction cost of \(x^*\) for some \(x^* \in \{0,\ldots,B\}\) and \(\gamma\in\Gamma\) that has not been found. 
	Let \(y=(\min\{p^1,q^1\},\min\{p^2,q^2\})\) with \(p=\VAL(H_1(\gamma^1),u)\) and \(q=\VAL(H_2(\gamma^2),u)\), where \(\gamma^1+\gamma^2=\gamma\) and let \(c=\sum_{a\in A}c(a)\cdot \gamma^1_a\). As described above, if \(p\in L(H_1,c)\) and \(q \in L(H_2,x^*-c)\), then \(y\) would have been added to \(L(H,x^*)\) due to construction of the algorithm. Thus, we may assume that \(p\notin L(H_1,c)\) or \(q\notin L(H_2,x^*-c)\). Without loss of generality, we assume that \(p\notin L(H_1,c)\). It follows that there exists a non-dominated point \(r \in L(H_1,c)\) with \(r\leq p\). Consequently, it holds that \((\min\{r^1,q^1\},\min\{r^2,q^2\})\leqq y\), which either contradicts the assumption that \(y\) is non-dominated or  the fact that \(y\) has not been found.\qed
\end{proof}

\begin{corollary}\label{cor:dynproalg}
	Let \(G\) be a two-terminal series-parallel graph and let \(T_G\) be its decomposition tree. After execution of the dynamic programming algorithm the following holds for all \(H\) in \(T_G\) and for all \(x\in\{0,\ldots,B\}\):
	\begin{itemize}
		\item If \(H\) is the parallel composition of \(H_1\) and \(H_2\), then for all \(p\in L(H,x)\) there exists \(r\in L(H_1,k)\) and \(s\in L(H_2,x-k)\) for some \(k\in \{0,\ldots,x\}\) with \(p=r+s\).
		\item If \(H\) is the series composition of \(H_1\) and \(H_2\), then for all \(p\in L(H,x)\) there exists \(r\in L(H_1,k)\) and \(s\in L(H_2,x-k)\) for some \(k\in \{0,\ldots,x\}\) with \(p=(\min\{r^1,s^1\},\min\{r^2,s^2\})\), where \(r=(r^1,r^2)\) and \(s=(s^1,s^2)\).
	\end{itemize}\qed
\end{corollary}

Further, the above described dynamic programming algorithm runs in pseudo-polynomial time.
\begin{theorem}
	The dynamic programming algorithm has a worst-case running-time complexity of \(\mathcal{O}(B^2m^3U^2\log(BmU))\).
\end{theorem}
\begin{proof}
	First, note that the size of \(L(H,x)\) is bounded from above by \(mU\), i.e., \(|L(H,x)| \leq mU\) for all \(x\in\{0,\ldots,B\}\) and for any \(H\) in \(T_G\). Further, the decomposition tree \(T_G\) has \(2m-1\) vertices containing \(m\) leaf vertices. The set of non-dominated points for a leaf vertex can be computed in constant time. Thus, in total \(\mathcal{O}(m)\) work is involved for all leaf vertices. For each of the remaining \(m-1\) vertices, we have to create at most \[\sum_{x=0}^{B}\sum_{k=0}^{x} |L(H_1,k)| \cdot |L(H_2,x-k)| \in \mathcal{O}(B^2m^2U^2)\] labels regardless of \(H\) being a series or a parallel composition of \(H_1\) and \(H_2\), respectively. Thus, the total number of labels created is in \(\mathcal{O}(B^2m^3U^2)\). Further, we have to check these label sets for non-dominance, which can be done in \(\mathcal{O}(B^2m^3U^2 \log(B^2m^3U^2)) = \mathcal{O}(B^2m^3U^2\log(BmU))\), cf. \cite{kung1975finding}. In total, a running-time complexity of \(\mathcal{O}(B^2m^3U^2\log(BmU))\) is achieved, which concludes the proof.\qed
\end{proof}

\begin{remark}
	Note that the worst-case running-time complexity in case of \(c(a)=1\) for all \(a\in A\) reduces to \(\mathcal{O}(m^5U^2\log(mU))\), since \(B\) can assumed to be smaller than \(m\). Further, the computation of \(L(H,x)\) in case of \(H\) being a primitive graph simplifies to:
	\begin{equation}\label{eq:4}
	L(H,x) =\begin{cases}
	\{(u^1(a),u^2(a))\}, & \text{ if } x = 0\\
	\{(0,0)\}, & \text{ else.}
	\end{cases}
	\end{equation}
	Thus, the above described dynamic programming algorithm can analogously be defined in the case of unit interdiction costs, i.e., \(c(a)=1\) for all \(a\in A\).
\end{remark}

For the remainder of this article, we assume that \(c(a)=1\) for all \(a\in A\). 
Further, note that both attainable maximum flow values \(\VAL(G(\gamma),u^i)\) lie in the interval between \(0\) and \(mU^i\), \(i=1,2\), for all \(\gamma\in\Gamma\). With respect to this setting, we extend the previously presented dynamic programming algorithm to an FPTAS by partitioning the objective space into polynomially many subintervals. More specifically, we partition the range between \(0\) and \(mU^i\) into \(q^i\) intervals in the following way:
\begin{equation*}
[0,(1+\varepsilon)^0), [(1+\varepsilon)^0, (1+\varepsilon)^1),\ldots, [(1+\varepsilon)^{q^i-1}, (1+\varepsilon)^{q^i}),
\end{equation*}
where \(q^i\coloneqq \lceil\log_{1+\varepsilon}(mU^i)\rceil\) and \(\varepsilon>0\). Note that \(q^i \in \mathcal{O}(\frac{1}{\varepsilon}\log(mU^i))\).
Further, we define two different kinds of label sets over a graph \(H\) in \(T_G\), denoted by \(A_{\varepsilon}(H,x)\) and \(L_{\varepsilon}(H,x)\) for all \(x\in \{0,1,\ldots,B\}\). The label sets \(L_{\varepsilon}(H,x)\) are computed by using the auxiliary label sets \(A_{\varepsilon}(H,x)\) and, finally, denote an \(\varepsilon\)-approximation on the set of non-dominated points in \(H\), where the interdictor has an interdiction budget of \(x\). The labels in \(A_{\varepsilon}(H,x)\) are obtained as follows.
For the case of \(H=(V_H,A_H)\) being a primitive graph with \(V_H=\{s_H,t_H\}\) and \(A_H=\{a^*\}\), where \(a^*=(s_H,t_H)\), we compute \(A_{\varepsilon}(H,x)\) for all \(x \in \{0,1,\ldots,B\}\), in the following way:
\begin{equation}\label{eq:5}
A_{\varepsilon}(H,x) =\begin{cases}
\{((1+\varepsilon)^i, (1+\varepsilon)^j, \gamma')\}, & \text{ if } x = 0\\
\{(0,0, \hat{\gamma})\}, & \text{ else,}
\end{cases}
\end{equation}
where \(\gamma'=(0,\ldots,0)\) and \(\hat{\gamma}_a\) equals \(1\) if \(a\) equals \(a^*\) and \(0\) otherwise. Further, \(i\) and \(j\) are chosen to be the maximal possible indices such that \((1+\varepsilon)^i\leq u^1(a^*)\) and \((1+\varepsilon)^j\leq u^2(a^*)\), respectively, i.e., \(i \coloneqq \max\{k\in\mathbb{N}\mid (1+\varepsilon)^k\leq u^1(a^*)\}\) and \(j \coloneqq \max\{k\in\mathbb{N}\mid (1+\varepsilon)^k\leq u^2(a^*)\}\). In the following, we say a label \(l=(l^1, l^2, \gamma^l) \in A_{\varepsilon}(H,x)\) dominates another label \(q=(q^1, q^2, \gamma^q) \in A_{\varepsilon}(H, x)\) for some \(x\in \{0,1,\ldots,B\}\), if \((l^1, l^2) \leq (q^1, q^2)\). 

In case of \(H\) being the parallel composition of \(H_1\) and \(H_2\), we compute \(A_{\varepsilon}(H,x)\) as follows:
\begin{equation}\label{eq:6}
A_{\varepsilon}(H,x) = \min \left\{\bigcup_{k=0}^x A_{\varepsilon}(H_1,k) \oplus A_{\varepsilon}(H_2,x-k)\right\} \text{ for } x = 0,1,\ldots,B, 
\end{equation}
where \(r\oplus s\coloneqq ((1+\varepsilon)^i, (1+\varepsilon)^j, \gamma^r+\gamma^s)\) with 
\begin{align*}
&i\coloneqq\max\{k\in\mathbb{N}\mid (1+\varepsilon)^k\leq \VAL(H_1(\gamma^r),u^1)+\VAL(H_2(\gamma^s),u^1)\},\\
&j\coloneqq\max\{k\in\mathbb{N}\mid (1+\varepsilon)^k\leq \VAL(H_1(\gamma^r),u^2)+\VAL(H_2(\gamma^s),u^2)\}
\end{align*}
and \(R\oplus S \coloneqq\{r\oplus s\mid r\in R, s\in S\}\).

If \(H\) is the series composition of \(H_1\) and \(H_2\), we compute \(A_{\varepsilon}(H,x)\) as follows:
\begin{equation}\label{eq:7}
A_{\varepsilon}(H,x) = \min \left\{\bigcup_{k=0}^x A_{\varepsilon}(H_1,k) \odot A_{\varepsilon}(H_2,x-k)\right\} \text{ for } x = 0,1,\ldots,B,
\end{equation}
where \(r\odot s\coloneqq ((1+\varepsilon)^i, (1+\varepsilon)^j, \gamma^r+\gamma^s)\) with 
\begin{align*}
&i\coloneqq\max\{k\in\mathbb{N}\mid (1+\varepsilon)^k\leq \min\{\VAL(H_1(\gamma^r),u^1), \VAL(H_2(\gamma^s),u^1)\}\},\\
&j\coloneqq\max\{k\in\mathbb{N}\mid (1+\varepsilon)^k\leq \min\{\VAL(H_1(\gamma^r),u^2), \VAL(H_2(\gamma^s),u^2)\}\}
\end{align*}
and \(R\odot S \coloneqq \{r\odot s\mid r\in R, s\in S\}\).

Note that we simply add up interdiction strategies both in the case of a series and parallel composition. Using the auxiliary label sets \(A_{\varepsilon}(H,x)\), we compute \(L_{\varepsilon}(H,x)\) as follows. For each label \(l=(l^1,l^2,\gamma^l) \in A_{\varepsilon}(H,x)\), we compute \(l^*~\coloneqq~\VAL(H(\gamma^l),u)\) and put it into \(L_{\varepsilon}(H,x)\). Again, dominated labels get discarded afterwards.
Note that we could simply store this information additionally in \(A_{\varepsilon}(H,x)\). However, for the sake of clarity and better readability, it is useful to be able to refer to both kinds of labels.

\begin{observation}
	Let \(\varepsilon>0\). First, observe that \(|A_{\varepsilon}(H,x)| \geq |L_{\varepsilon}(H,x)|\) for all \(H\) in \(T_G\) and for all \(x \in \{0,1,\ldots,B\}\). Second, note that for all \(p\in L_{\varepsilon}(H, x)\) there exists a label \(q \in A_{\varepsilon}(H, x)\) with \(q = (q^1, q^2, \gamma')\) and \(p = (\VAL(H(\gamma'),u^1), \VAL(H(\gamma'),u^2))\) and \((q^1, q^2)\leqq p\) for all \(H\) in \(T_G\) and for all \(x\in\{0,\ldots,B\}\). 
\end{observation}

To prove that this procedure defines an FPTAS for (BMFNI) with unit interdiction costs, we need to show that the points in \(L_{\varepsilon}(G,B)\) define an \(\varepsilon\)-approximation of \(L(G,B)\) and that the size of the label sets throughout the execution of the algorithm is bounded by a polynomial of the size of the problem instance.

\begin{theorem}\label{lemma:correctness}
	Let \(\varepsilon>0\). After termination of the approximation scheme defined by formulas \eqref{eq:5}, \eqref{eq:6} and \eqref{eq:7}, the label set \(L_{\varepsilon}(G,B)\) defines an \(\varepsilon\)-approxi-mation on the set of non-dominated points of the \Problem\ instance.
\end{theorem}
\begin{proof}
	Again, we use induction on the size of the decomposition tree \(T_G\) of \(G\). Let \(H=(V_H,A_H)\) be a leaf of \(T_G\) with \(V_H=\{s_H,t_H\}\) and \(A_H=\{a^*\}\). We need to distinguish two cases, i.e., \(x=0\) and \(x>0\). For \(x=0\), let \(g\) be the unique label in \(A_{\varepsilon}(H,0)\) and let \(\gamma^g=(0,\ldots,0)\) be its interdiction strategy. Thus, we create the unique label \(q\coloneqq\VAL(H(\gamma^g),u)=u(a^*)\) at \(L_{\varepsilon}(H,0)\). Consequently, it holds for the unique label \(p \in L(H,0)\) with \(p=u(a^*)\) that \(q = u(a^*) \leqq (1+\varepsilon)u(a^*) = (1+\varepsilon)p\).
	For \(x>0\), let \(h\) be the unique label in \(A_{\varepsilon}(H,x)\) and let \(\gamma^h\) be its interdiction strategy with \(\gamma^h_{a}\) equals \(1\) if \(a\) equals \(a^*\) and \(0\) otherwise. Thus, we create the unique label \(q\coloneqq\VAL(H(\gamma^h),u)=(0,0)\) at \(L_{\varepsilon}(H,x)\). Consequently, it holds for the unique label \(p \in L(H,x)\) with \(p=(0,0)\) that \(q = (0,0) \leqq (1+\varepsilon)(0,0) = (1+\varepsilon)p\).
	
	Now, let \(H\) be the parallel composition of \(H_1\) and \(H_2\). Further, let \(x \in \{0,1,\ldots,B\}\) and \(p\in L(H,x)\) be a non-dominated point. By Corollary \ref{cor:dynproalg}, we know that there exist \(r\in L(H_1,k)\) and \(s\in L(H_2,x-k)\) for some \(k\in \{0,1,\ldots,x\}\) with \(r+s=p\). By induction hypothesis, we know that there exist \(r'\in L_{\varepsilon}(H_1, k)\) and \(s' \in L_{\varepsilon}(H_2, x-k)\) with \(r'\leqq (1+\varepsilon)r\) and \(s'\leqq(1+\varepsilon)s\). Due to construction, there exist \(r''\in A_{\varepsilon}(H_1, k)\) and \(s''\in A_{\varepsilon}(H_2, x-k)\) such that \(r' = \VAL(H_1(\gamma^{r''}),u)\) and \(s' = \VAL(H_2(\gamma^{s''}),u)\), where \(\gamma^{r''}\) and \(\gamma^{s''}\) denote the corresponding interdiction strategies of \(r''\) and \(s''\), respectively. Now, two cases might occur that need to be distinguished. Either there exists a label \(q = \VAL(H(\gamma^*),u) \in L_{\varepsilon}(H,x)\) with \(\gamma^* \coloneqq \gamma^{r''} + \gamma^{s''}\) or there exists a label \(y = \VAL(H(\hat{\gamma}),u) \in L_{\varepsilon}(H,x)\) with \(y\leqq q\) for some \(\hat{\gamma}\in \Gamma\). Since \(y\leqq q\), we only have to consider the former case. Therefore, let us assume that
	\(q \in L_{\varepsilon}(H,x)\). Thus, it holds \(q=\VAL(H(\gamma^*),u) = \VAL(H_1(\gamma^{r''}),u) + \VAL(H_2(\gamma^{s''}),u) = r' + s' \leqq (1+\varepsilon)r + (1+\varepsilon)s = (1+\varepsilon)(r+s) = (1+\varepsilon)p\). 
	
	Now, let \(H\) be the series composition of \(H_1\) and \(H_2\). Again, let \(x \in \{0,\ldots,B\}\) and \(p\in L(H,x)\) be a non-dominated point. By Corollary \ref{cor:dynproalg}, we know that there exist \(r\in L(H_1,k)\) and \(s\in L(H_2,x-k)\) for some \(k\in \{0,1,\ldots,x\}\) with \((\min\{r^1, s^1\}, \min\{r^2, s^2\}) = p\). By induction hypothesis, we know that there exist \(r'\in L_{\varepsilon}(H_1, k)\) and \(s' \in L_{\varepsilon}(H_2, x-k)\) with \(r'\leqq (1+\varepsilon)r\) and \(s'\leqq(1+\varepsilon)s\). Due to construction, there exist \(r''\in A_{\varepsilon}(H_1, k)\) and \(s''\in A_{\varepsilon}(H_2, x-k)\) such that \(r' = \VAL(H_1(\gamma^{r''}),u)\) and \(s' = \VAL(H_2(\gamma^{s''}),u)\), where \(\gamma^{r''}\) and \(\gamma^{s''}\) denote the corresponding interdiction strategies of \(r''\) and \(s''\), respectively. Again, two cases might occur that need to be distinguished. Either there exists a label \(q = \VAL(H(\gamma^*),u) \in L_{\varepsilon}(H,x)\) with \(\gamma^* \coloneqq \gamma^{r''} + \gamma^{s''}\) or there exists a label \(y = \VAL(H(\hat{\gamma}),u) \in L_{\varepsilon}(H,x)\) with \(y\leqq q\) for some \(\hat{\gamma}\in \Gamma\). Since \(y\leqq q\), we only have to consider the first case. Therefore, let us assume that
	\(q \in L_{\varepsilon}(H,x)\). Thus, it holds that
	\begin{align*}
	q&=\VAL(H(\gamma^*),u)\\
	&= (\min\{\VAL(H_1(\gamma^{r''}),u^1), \VAL(H_2(\gamma^{s''}),u^1)\},\\ &\min\{\VAL(H_1(\gamma^{r''}),u^2), \VAL(H_2(\gamma^{s''}),u^2)\})\\
	&= (\min\{r^{'1}, s^{'1}\}, \min\{r^{'2}, s^{'2}\})\\
	& \leqq (\min\{(1+\varepsilon)r^1, (1+\varepsilon)s^1\}, \min\{(1+\varepsilon)r^2, (1+\varepsilon)s^2\})\\
	&= (1+\varepsilon)(\min\{r^1,s^1\}, \min\{r^2,s^2\}) = (1+\varepsilon)p,
	\end{align*}
	which concludes the proof.\qed
\end{proof}

\begin{corollary}\label{cor:approximation}
	Let \(G\) be a two-terminal series-parallel graph and let \(T_G\) be its decomposition tree. After execution of the approximation scheme for some \(\varepsilon~>~0\), the label set \(L_{\varepsilon}(H,x)\) defines an \(\varepsilon\)-approximation on the set \(L(H,x)\) for all \(H\) in \(T_G\) and for all \(x\in\{0,\ldots,B\}\).\qed
\end{corollary}

\begin{theorem}\label{thm:running}
	The time involved for computing \(L_{\varepsilon}(G,B)\) is in 
	\begin{equation*}
	\mathcal{O}(m^3\frac{1}{\varepsilon^2}\log^2(mU)\log(m\log(mU))+ \mathcal{T}m^3\frac{1}{\varepsilon^2}\log^2(mU)),
	\end{equation*}
	where \(\mathcal{T}\) denotes the time for solving a maximum flow problem on a two-terminal series-parallel graph.
\end{theorem}
\begin{proof}
	First, note that for all \(x\in \{0,\ldots,B\}\) and for all \(H\) in \(T_G\) the size of the auxiliary label set \(A_{\varepsilon}(H,x)\) is bounded from above by \(\lceil\log_{1+\varepsilon}(mU)\rceil\), which is in \(\mathcal{O}(\frac{1}{\varepsilon}\log(mU))\). Further, the decomposition tree \(T_G\) has \(2m-1\) vertices containing \(m\) leaf vertices. The set of non-dominated points in \(A_{\varepsilon}(H,x)\) for a leaf vertex can be computed in constant time. Thus, in total \(\mathcal{O}(m)\) work is involved for all leaf vertices. For each of the remaining \(m-1\) vertices, we have to create at most \[\sum_{x=0}^{B}\sum_{k=0}^{x} |A_{\varepsilon}(H_1,k)| \cdot |A_{\varepsilon}(H_2,x-k)| \in \mathcal{O}(B^2\frac{1}{\varepsilon^2}\log^2(mU))= \mathcal{O}(m^2\frac{1}{\varepsilon^2}\log^2(mU))\] labels regardless of \(H\) being the series or a parallel composition of \(H_1\) and \(H_2\), respectively. Thus, the total number of labels created is in \(\mathcal{O}(m^3\frac{1}{\varepsilon^2}\log^2(mU))\). Two maximum flow problems have to be solved for each of those labels, which is in  \(\mathcal{O}(\mathcal{T}m^3\frac{1}{\varepsilon^2}\log^2(mU)))\). Moreover, we have to check these label sets for non-dominance, which can be done in \(\mathcal{O}(m^3\frac{1}{\varepsilon^2}\log^2(mU)\log(m^3\frac{1}{\varepsilon^2}\log^2(mU))) = \mathcal{O}(m^3\frac{1}{\varepsilon^2}\log^2(mU)\log(m\frac{1}{\varepsilon}\log(mU)))\), cf. \cite{kung1975finding}. Using the label set \(A_{\varepsilon}(G,B)\), we can compute \(L_{\varepsilon}(G,B)\) in \(\mathcal{O}(\mathcal{T}\frac{1}{\varepsilon}\log(mU))\), which is in \(\mathcal{O}(\mathcal{T}m^3\frac{1}{\varepsilon^2}\log^2(mU)))\) and can thus be neglected.  In total, we get a running-time complexity of \[\mathcal{O}(m^3\frac{1}{\varepsilon^2}\log^2(mU)\log(m\frac{1}{\varepsilon}\log(mU))+ \mathcal{T}m^3\frac{1}{\varepsilon^2}\log^2(mU)),\] which concludes the proof.\qed
\end{proof}

\begin{corollary}
	There is an FPTAS for \Problem\ on two-terminal series-parallel graphs with unit interdiction costs constructing an \(\varepsilon\)-approximation on the set of non-dominated points.
\end{corollary}
\begin{proof}
	Follows from Lemma \ref{lemma:correctness} and Lemma \ref{thm:running}.\qed
\end{proof}

\section{Conclusion}\label{sec:conclusion}
We introduced the biobjective maximum flow network interdiction problem, called \Problem. We showed that \Problem\ is intractable and that the decision problem, which asks whether a feasible interdiction strategy is efficient, is \(\mathcal{NP}\)-complete, even on graphs with only two vertices and several parallel arcs connecting those. We showed that \Problem\ on those graphs can be formulated as a biobjective knapsack problem and can thus be approximated within arbitrary precision. Lastly, we proposed a dynamic programming algorithm to solve \Problem\ on two-terminal series-parallel graphs, which we extended in case of unit interdiction costs to a fully polynomial-time approximation scheme and thus, addressing the gap of approximation schemes in network interdiction problems. In the future, it would be interesting to investigate whether \Problem\ can be approximated on two-terminal series-parallel graphs with arbitrary interdiction costs, or even on general directed graphs.

\begin{acknowledgements}
This work was partially supported by the Bundesministerium f\"ur Bildung und Forschung (BMBF) under Grant No. 13N14561 and Deutscher Akademischer Austauschdienst (DAAD) under Grant No. 57518713. Carlos M. Fonseca acknowledges national funding through the FCT - Foundation for Science and Technology, I.P., in the scope of bilateral cooperation project ``Multiobjective Network Interdiction'' and CISUC - UID/CEC/00326/2020, and the European Social Fund, through the Regional Operational Program Centro 2020.
\end{acknowledgements}

%
%



\end{document}